\documentclass[hyperref]{amsart}
\usepackage{color}
\usepackage{booktabs}
\usepackage{longtable}
 
\usepackage{amscd}
\usepackage{amssymb}
\usepackage{amsmath}
\usepackage{xypic}
\usepackage{nicefrac}
\usepackage{color}
\usepackage{tikz}

\usepackage{multirow}
\usepackage{bigdelim}
\usepackage{hyperref}

\newtheorem{theorem}{Theorem}[section]
\newtheorem*{ttheorem}{Theorem}

\newtheorem{proposition}[theorem]{Proposition}

\newtheorem{corollary}[theorem]{Corollary}
{\theoremstyle{remark}
\newtheorem{remark}[theorem]{Remark}

}

\newcommand{\savefootnote}[2]{\footnote{\label{#1}#2}}
\newcommand{\repeatfootnote}[1]{\textsuperscript{\ref{#1}}}

\theoremstyle{definition}

\newtheorem{definition}[theorem]{Definition}
\newtheorem{example}[theorem]{Example}

\newcommand{\T}{{\mathcal{T}}}

\newcommand{\CC}{\mathbb{C}}
\renewcommand{\L}{\mathcal{L}}
\newcommand{\RR}{\mathbb{R}}

\newcommand{\QQ}{\mathbb{Q}}
\newcommand{\ZZ}{\mathbb{Z}}
\newcommand{\NN}{\mathbb{N}}
\newcommand{\E}{\mathfrak{E}}

\newcommand{\PP}{\mathbb{P}}
\newcommand{\A}{\mathbb{A}}

\newcommand{\CO}{{\mathcal{O}}}

\newcommand{\cA}{{\mathcal{A}}}
\newcommand{\X}{{\mathcal{X}}}

\newcommand{\D}{{\mathfrak{D}}}

\DeclareMathOperator{\TV}{TV}

\DeclareMathOperator{\DF}{DF}

\DeclareMathOperator{\face}{face}

\DeclareMathOperator{\rk}{rank}

\DeclareMathOperator{\tail}{tail}

\DeclareMathOperator{\wdiv}{Div}

\DeclareMathOperator{\spec}{Spec}
\DeclareMathOperator{\proj}{Proj}
\DeclareMathOperator{\ord}{ord}
\DeclareMathOperator{\supp}{supp}

\DeclareMathOperator{\conv}{conv}

\DeclareMathOperator{\vol}{vol}

\DeclareMathOperator{\Aut}{Aut}

\DeclareMathOperator{\Div}{div}
\DeclareMathOperator{\discr}{discr}

\def\poldivdpi{
 \begin{tikzpicture}[scale=0.4]
    \draw[dotted,step=1,gray] (-2,-2) grid (2,2);
    \draw[dashed,gray] (-2,0) -- (2,0);
    \draw[dashed,gray] (0,-2) -- (0,2);
   \draw (-1,1) -- (0,1) -- (1,0);
  \end{tikzpicture}
}
\def\poldivdpii{
  \begin{tikzpicture}[scale=0.4]
    \draw[dotted,step=1,gray] (-2,-2) grid (2,2);
    \draw[dashed,gray] (-2,0) -- (2,0);
    \draw[dashed,gray] (0,-2) -- (0,2);
   \draw (-1,-1) -- (0,0) -- (1,0);
  \end{tikzpicture}
}
\def\poldivdpiii{
  \begin{tikzpicture}[scale=0.4]
    \draw[dotted,step=1,gray] (-2,-2) grid (2,2);
    \draw[dashed,gray] (-2,0) -- (2,0);
    \draw[dashed,gray] (0,-2) -- (0,2);
    \draw (-1,0) -- (0,1) -- (1,0);
  \end{tikzpicture}
}
\def\poldivdpiv{
  \begin{tikzpicture}[scale=0.4]
    \draw[dotted,step=1,gray] (-2,-2) grid (2,2);
    \draw[dashed,gray] (-2,0) -- (2,0);
    \draw[dashed,gray] (0,-2) -- (0,2);
    \draw (-1,0) -- (0,1) -- (1,0);
  \end{tikzpicture}
}

\def\polytopedp{\begin{tikzpicture}[scale=0.4]
    \draw[dotted,step=1,gray] (-2,-2) grid (2,2);
    \draw[dashed,gray] (-2,0) -- (2,0);
    \draw[dashed,gray] (0,-2) -- (0,2);
    \draw (-1,0) -- (0,-1) -- (1,0);
    \draw (-1,0) -- (0,1) -- (1,0);
\end{tikzpicture}
}

\def\poldivdp{
 \poldivdpi\hspace{0.3cm}
 \poldivdpii\hspace{0.3cm}
 \poldivdpiii
}

\def\poldivdpsum{
 \raisebox{-0.7cm}{\poldivdpi} + \raisebox{-0.7cm}{\poldivdpii} = \raisebox{-0.7cm}{\poldivdpiv}
}

\def\poldivdpcombine{
 \raisebox{-0.7cm}{\poldivdpiii} \scalebox{1}[-1]{$\curvearrowleft$} \raisebox{-0.7cm}{\poldivdpiv} $\leadsto$ \raisebox{-0.7cm}{\polytopedp}
}

\title{K-stability for Fano manifolds with torus action of complexity one}

\author[N. Ilten]{Nathan Ilten}

\author[H. S\"u{\ss}]{Hendrik S\"u\ss}
\address{Nathan Ilten\\
Deparment of Mathematics,
Simor Fraser University,
8888 University Drive,
Burnaby BC V5A1S6 Canada}
\email{\href{mailto:nilten@sfu.ca}{nilten@sfu.ca}}

\address{Hendrik S\"u\ss\\ School of Mathematics,
The University of Manchester,
Alan Turing Building,
Oxford Road
Manchester M13 9PL}
\email{\href{mailto:hendrik.suess@manchester.ac.uk}{hendrik.suess@manchester.ac.uk}}

\thanks{The authors thank the International Centre for Mathematical Sciences for partial support}

\begin{document}
\maketitle
\begin{abstract}
We consider Fano manifolds admitting an algebraic torus action with general orbit of codimension one. Using a recent result of 
Datar and Sz\'ekelyhidi, we effectively determine the existence of K\"ahler-Ricci solitons for those manifolds via the notion of \emph{equivariant} K-stability. This allows us to give new examples of K\"ahler-Einstein Fano threefolds, and Fano threefolds admitting a non-trivial K\"ahler-Ricci soliton.
\end{abstract}

\section{Introduction}
Thanks to the proof of the Yau-Tian-Donaldson conjecture for Fano manifolds
\cite{kestab,zbMATH06394344,zbMATH06394345,zbMATH06394346, tian15} there is a completely algebraic characterization for a Fano manifold $X$ to admit a K\"ahler-Einstein metric through the notion of K-stability. In concrete cases, however, this criterion is far from being effective. Indeed, one has to check the positivity of the Donaldson-Futaki invariant for \emph{all} possible degenerations.

In a recent paper, Datar and Sz\'ekelyhidi proved that given the action of a reductive group $G$ on $X$, it suffices to consider only equivariant degenerations. For certain classes of varieties this \emph{equivariant} version of K-stability can be checked effectively. For example, for toric varieties  there are no degenerations beside the trivial one left to consider. Hence, one recovers the result of Wang and Zhu \cite{wang04}.

In this paper we consider a generalization of toric Fano varieties, namely, Fano $T$-varieties of complexity one. These are varieties admitting an algebraic torus action with maximal orbits of codimension one. Similar to toric varieties, these varieties come with a combinatorial description. We describe the equivariant degenerations in terms of the corresponding combinatorial data and calculate the corresponding Donaldson-Futaki invariant. In particular, we show that there are only finitely many degenerations to consider; this leads to an effective criterion for equivariant K-stability and hence the existence of a K\"ahler-Einstein metric. Applying this to the combinatorial description of Fano threefolds with 2-torus action from \cite{t3folds} yields the following theorem:

\begin{ttheorem}[See Theorem \ref{thm:threefold}]
  The Fano threefolds Q, 2.24\savefootnote{one-of-family}{This refers only to a particular element of the family admitting a 2-torus action}, 2.29, 2.32, 3.10\repeatfootnote{one-of-family}, 3.19, 3.20, 4.4 and 4.7 from  Mori and Mukai's classification \cite{mori-mukai} are K\"ahler-Einstein.
\end{ttheorem}

\noindent To our knowledge, the examples 2.29, 3.19. 3.20, 4.4, and 4.7 are the first cases where the existence of a K\"ahler-Einstein metric has been verified via K-stability and not any other method.

We now sketch how equivariant degenerations and their Donaldson-Futaki invariants may be described for a complexity-one Fano $T$-variety.
Roughly speaking, a polarized $T$-variety of complexity one corresponds to the combinatorial data of a finite set of piecewise linear functions $\Psi_i$ over some lattice polytope $\Box$ in a vector space $M_\RR$. For example, the functions over the line segment $\Box=[-1,1] \subset \RR$ with graphs sketched in Figure~\ref{fig:dp1} correspond to a singular del Pezzo surface of degree $4$.
\begin{figure}[h]
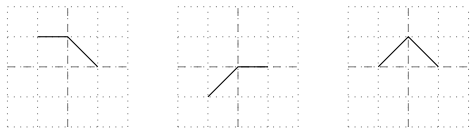

  \centering
  \poldivdp
  \caption{Combinatorial data for a singular del Pezzo surface:  piecewise linear functions $\Psi_1$, $\Psi_2$ and $\Psi_3$ on $[-1,1]$.}
  \label{fig:dp1}
\end{figure}

We prove that equivariant degenerations of this $T$-variety essentially correspond to the choice of one of these functions say $\Psi_j$. Its graph serves as the upper boundary of a polytope $\Delta_j \subset M_\RR \times \RR$, which describes the toric special fiber. The lower boundary of $\Delta_j$ is formed by the graph of the negative sum of the remaining functions. For example, if we pick the third of the three functions from Figure~\ref{fig:dp1}, then the graph of the sum of the remaining functions is sketched in Figure~\ref{fig:dpsum}. Now, the polytope on the right hand side of Figure~\ref{fig:dpglue} is bounded from above by $\Psi_3$ and from below by $-(\Psi_1 + \Psi_2)$.
\begin{figure}[h]
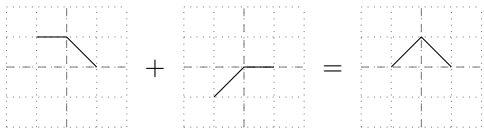

  \centering
  \poldivdpsum
  \caption{Summing up all but one of the functions: $\Psi_1 + \Psi_2$}
 \label{fig:dpsum}
\end{figure}

\begin{figure}[h]
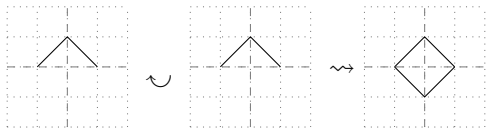

  \centering
  \poldivdpcombine
\caption{Combining the sum with the remaining function gives the polytope $\Delta_3$ of the toric special fiber.}
\label{fig:dpglue}
\end{figure}
The Donaldson-Futaki invariant for this degeneration is just the projection of the barycenter of $\Delta_j$ to the second component of $M_\RR \times \RR$. In our example obviously the barycenter of $\Delta_3$ is $(0,0)$. Hence, the Donaldson-Futaki invariant vanishes and the corresponding del Pezzo surface is not K-stable.

We organise the paper as follows. In Section~\ref{sec:kstab} we recall the definition of (equivariant) K-stability for the case of Fano varieties $X$ and pairs $(X,\xi)$, were $\xi$ is a vector field. In Section~\ref{sec:tvars} we recapitulate the combinatorial description of $T$-varieties of complexity one from \cite{t-survey}. In Section~\ref{sec:test-configs} we describe the equivariant degenerations in terms of the combinatorial data and calculate the corresponding Donaldson-Futaki invariant. In Section~\ref{sec:addit-symmetries} we consider the action of the normalizer of the torus in $\Aut(X)$ in order to further reduce the number
of degenerations to consider. In certain cases we are left with only the trivial degenerations. In this way, we obtain the existence of non-trivial K\"ahler-Ricci solitons for some Fano threefolds:
\setcounter{footnote}{0}
\begin{ttheorem}[See Theorem~\ref{thm:soliton}]
  The Fano threefolds 3.8\savefootnote{one-of-family3}{This refers only to a particular element of the family admitting a 2-torus action}, 3.21, 4.5\repeatfootnote{one-of-family3} from Mori-Mukai's list \cite{mori-mukai} admit a non-trivial K\"ahler-Ricci soliton.
\end{ttheorem}
\noindent To our knowledge, these are the only known examples of non-trivial K\"ahler-Ricci solitons which are neither toric nor the solution of an ODE.

\section{Equivariant K-stability}
\label{sec:kstab}
The key objects involved in defining K-stability are test configurations: 
\begin{definition}
Let $(X,L)$ be a polarized projective variety.
  A \emph{test configuration} for $(X,L)$ is a $\CC^*$-equivariant flat family $\X$ over $\A^1$ equipped with a relatively ample equivariant $\QQ$-line bundle $\L$ such that 
\begin{enumerate}
\item The $\CC^*$-action $\lambda$ on $(\X,\L)$ lifts the standard $\CC^*$-action on $\A^1$;
\item The general fiber is isomorphic to $X$, with $\L$ restricting to $L$.
\end{enumerate}
A test configuration with $\X \cong X \times \A^1$ is called a \emph{trivial configuration}. A test configuration with normal special fiber $\X_0$ is called \emph{special}. 
Given an algebraic group $G$ with action on $X$, a test configuration is called \emph{$G$-equivariant} if the action extends to $(\X,\L)$ and commutes with the $\CC^*$-action of the test configuration.
\end{definition}
\begin{remark}
The $\CC^*$-action guarantees that the family $\X$ is trivial over $\A^1\setminus\{0\}$.
\end{remark}
\begin{remark}
 By Hironaka's Lemma \cite[III.9.12]{MR0463157} the normality of the fibers induces the normality of the total space of the family. Hence, a special test configuration $(\X,\L)$ has normal total space $\X$.
\end{remark}

We will primarily be concentrating on the situation where $X$ is Fano and $L=\CO(- K_X)$. It follows that the special fiber $\X_0$ is $\QQ$-Fano. 
 We proceed to define the (modified) Donaldson-Futaki invariants as they appeared in \cite{berman14}.

Let $\X_0$ be any $\QQ$-Fano variety equipped with the action of an algebraic torus $T'$, and let $\ell\in \NN$ be the smallest natural number such that $-\ell K_{\X_0}$ is Cartier. Let $M'$ and $N'$ be the lattices of characters and one-parameter subgroups of $T'$, respectively. We denote the associated $\QQ$- and $\RR$-vector spaces by $M'_\QQ$, $M'_\RR$, etc. Fix an element $\xi\in N'_\RR$.
 Consider the canonical linearization for $\L_0 = \CO(-\ell K_{\X_0})$ coming from the identification $\L_0 \cong (\bigwedge^{\dim X} \T_{\X_0})^{\otimes \ell}$. Then we set $l_k=\dim H^0(X,\L_0^{\otimes k})$ and for every $v \in N'$
\[w_k(v)= \sum_{u \in M'} \langle u,v\rangle \cdot e^{\nicefrac{\langle u, \xi \rangle}{k}}\cdot \dim H^0(X,\L_0^{\otimes k})_{u}.\]
Now, 
\begin{equation}
  \label{eq:3}
  F_{\X_0,\xi}(v):=-\lim_{k \to \infty} \frac{w_k(v)}{k\cdot l_k\cdot \ell},
\end{equation}
defines a linear form on $N_\RR'$, i.e. an element of $M_\RR'$. We set $F_{\X_0} :=  F_{\X_0,0}$ for the special case $\xi=0$.

Consider now a Fano variety $X$ with a (possibly trivial) torus $T$ acting on it. Fix  $\xi \in N_\RR$, where $N$ is the lattice of one-parameter subgroups of $T$. Let $\X_0$ be the central fiber of a special $T$-equivariant test configuration $\X$ for $(X,\CO(-K_X))$. Then $\X_0$ comes equipped with a $T'=T\times \CC^*$-action induced by the test configuration. Let $v\in N'$ denote the one-parameter subgroup corresponding to the $\CC^*$-action of the test configuration. Note that the inclusion $T\hookrightarrow T'$ induces an inclusion of one-parameter subgroups $N\hookrightarrow N'$.

\begin{definition}[cf. \cite{berman14}]
 Given a test configuration $(\X,\L)$ of $(X,\CO(-K_X))$ as above,
  its \emph{Donaldson-Futaki invariant} is defined as \[\DF(\X, \L)=F_{\X_0}(v).\] Its \emph{modified Donaldson-Futaki invariant} is defined as \[\DF_\xi(\X,\L)=F_{\X_0,\xi}(v).\]
\end{definition}

\begin{definition}
  Consider a Fano variety $X$ with action by a reductive group $G$ containing a maximal torus $T \subset G$ and $\xi \in N_\RR$. The pair $(X,\xi)$ is called \emph{equivariantly K-stable} if $\DF_\xi(\X,\L) \geq 0$ for every $G$-equivariant special test configuration $\X$ as above and we have equality exactly in the case of trivial configurations. If $\xi=0$, we simply say that $X$ is equivariantly K-stable.
\end{definition}

The following result by Gabor Sz\'ekelyhidi and Ved Datar motivates the study of \emph{equivariant} K-stability:
\begin{theorem}[{\cite{datar15}}]
\label{thm:gabor}
For a Fano $G$-variety $X$, the variety  $X$ admits a K\"ahler-Ricci soliton if and only if there is an $\xi\in N_\RR$ such that the pair $(X,\xi)$ is equivariantly K-stable. In particular, $X$ is K\"ahler-Einstein if and only if $X$ is equivariantly K-stable.
\end{theorem}

\section{Polyhedral divisors and divisorial polytopes}
\label{sec:tvars}
In this section, we recall notions related to \emph{polyhedral divisors} \cite{affine-tvar} and \emph{divisorial polytopes} \cite{proj-tvar} which are useful for studying varieties with torus action.  See also \cite{t-survey} for an introduction to these concepts. 

\subsection{Polyhedral Divisors}\label{sec:pdiv}
Let $Y$ be a normal semiprojective variety and let $\sigma\subset N_\RR$ be a pointed, rational polyhedral cone. By $\sigma^\vee$ we denote the dual cone of $\sigma$.
 \begin{definition}A \emph{polyhedral divisor} on $(Y,N)$ with tail cone $\sigma$ is a 
formal finite sum
$$\D = \sum_P \D_P \otimes P,$$
where $P$ runs over all prime divisors on $Y$ and $\D_P$ is a rational polyhedron with tailcone $\sigma$.
Here, finite means that only finitely many coefficient differ from the 
tail cone. 
If the lattice $N$ is clear, we often speak simly of polyhedral divisors on $Y$.
If $Y$ is a complete curve, we define the \emph{degree} of a polyhedral divisor by $$\deg \D := \sum_P \Delta_P $$ where summation is via Minkowski addition. 
\end{definition}
Each polyhedron $\D_P$ gives rise to a map
\begin{align*}
\D_P:M\cap \sigma^\vee &\to \QQ\\
u&\mapsto \D_P(u):=\min_{v\in\D_P} \langle v,u\rangle.
\end{align*}
We can evaluate a polyhedral divisor for every element $u \in \sigma^\vee \cap M$ via
$$\D(u):=\sum_P \D_P(u)\cdot  P$$
to get a $\QQ$-divisor on $Y$.

A polyhedral divisor $\D$ is called a \emph{p-divisor} if $\D(u)$ is $\QQ$-Cartier and semiample for every $u\in \sigma^\vee\cap M$, and for some such $u$ the evaluation $\D(u)$ is big, that is, some multiple has a section with affine complement. In this case, the ring
\[
\bigoplus_{u\in\sigma^\vee \cap M} H^0(Y,\CO(\lfloor \D(u)\rfloor))
\]
is finitely generated and its spectrum $\TV(\D)$ is a normal affine of dimension $\dim Y+\rk N$ \cite[Theorem 3.1]{affine-tvar}. The  grading induces an effective action on $\TV(\D)$ by the torus $T$ whose character lattice is $M$.
Conversely, every normal affine $T$-variety is equivariantly isomorphic to $\TV(\D)$ for some p-divisor $\D$ on a semiprojective $Y$ \cite[Theorem 3.4]{affine-tvar}. Note, that by construction the rational functions $\CC(Y)$ equal the invariant rational functions $\CC(X)^T$ and this inclusion defines a rational quotient map $\pi:X \dashrightarrow Y$.

We will be especially interested in rational affine $T$-varieties $X$ with a complexity-one torus action having only constant invariant global functions. Such an $X$ may always be described by a p-divisor on $\PP^1$. On the other hand, the criterion for checking that a polyhedral divisor on $\PP^1$ is a p-divisor is especially simple: $\D$ is a p-divisor if and only if $\deg \D$ is properly contained in $\sigma$ \cite[Example 2.12]{affine-tvar}.

\subsection{Invariant Divisors}
Let $\D$ be a p-divisor on $Y$. Torus invariant divisors on $X=\TV(\D)$ have been studied in \cite{petersen}. Prime invariant divisors may be classified into two types: \emph{vertical divisors}, which arise as closures of a union of $k$-dimensional orbits, and \emph{horizontal divisors}, which arise as closures of a union of $k-1$-dimensional orbits. Here, $k$ denotes the dimension of the torus $T$.

We will primarily be interested in the case when $Y$ is a smooth projective curve. In this case, vertical divisors are in bijection with pairs $(P,v)$ where $P\in Y$ and $v$ is a vertex of $\D_P$; we denote the corresponding divisor by $D_{P,v}$. Likewise, horizontal divisors are in bijection with so-called \emph{extremal rays}: rays $\rho$ of $\sigma$ such that $\sigma\cap \deg \D=\emptyset$. We denote the corresponding divisor by $D_\rho$.

The invariant principal divisor associated to a semi-invariant function $f\cdot \chi^u$, $f\in \CC(Y)$ is calculated as follows (cf. \cite[Proposition 3.13]{petersen}):
\begin{equation}
  \label{eq:principal}
  \Div(f\cdot \chi^u)=\sum_\rho \rho(u)\cdot D_\rho+\sum_{P,v} \mu(v)(\langle u,v\rangle +\ord_P f)\cdot D_{P,v}.
\end{equation}

Here, $\rho(u)$ denotes the value of $u$ on the primitive generator of $\rho$, and $\mu(v)$ denotes the smallest natural number such that $\mu(v)v\in N$.

Moreover, a canonical divisor on $X$ is given by
\begin{equation}
  \label{eq:canonical}
  -\sum_\rho D_\rho + \sum_{P,v} (1-\mu(v))\cdot D_{P,v}  + \sum_{P,v} \mu(v) a_P  \cdot D_{P,v},
\end{equation}
where $K_{\PP^1} = \sum_P a_P \cdot P$ is a canonical divisor on $\PP^1$.
\subsection{Divisorial Polytopes}
\label{sec:divisorial-polytopes}
Just as polarized toric varieties correspond to lattice polytopes, polarized rational complexity-one $T$-varieties correspond to so-called divisorial polytopes:
\begin{definition}\cite[\S 3]{proj-tvar}
\label{def:divpol}
Let $\Box\subset M_\RR$ be a lattice polytope. 
  A \emph{divisorial polytope} over $\Box$ is a continuous piecewise affine concave function 
\[\Psi=\sum_{P\in \PP^1} \Psi_P\otimes P: \Box\cap M_\QQ \rightarrow \wdiv_\QQ \PP^1,\] such that:
\begin{enumerate}
\item For every $u$ in the interior of $\Box$, $\deg \Psi(u) > 0$;
\item For $P\in \PP^1$, the graph of $\Psi_P$ has integral vertices. 
\end{enumerate}
The \emph{support} of a divisorial polytope is the set of points $P \in \PP^1$ where $\Psi_P \not \equiv 0$. We will denote it by $\supp \Psi$.
\end{definition}
The polarized $T$-variety $(X,L)$ corresponding to $\Psi$ may be recovered via the equality
\[
H^0(X,L^{\otimes k})=\bigoplus_{u\in \Box\cap \frac{1}{k}M} H^0(Y,\CO(\lfloor k\cdot \Psi(u)\rfloor)).
\]
On the other hand, the affine cone
\[
\spec \bigoplus_i H^0(X,L^{\otimes i})
\]
corresponds to a p-divisor $\D$ with respect to the lattice $N\times \ZZ$ as described above. Indeed, by \cite[Prop. 2.11]{affine-tvar} the mapping
\[(u,k) \mapsto \D((u,k))= k\cdot \Psi(\nicefrac{1}{k} \cdot u)\]
defines a p-divisor $\D$ with tail cone given by 
\[
(\tail \D)^\vee = \QQ_{\geq 0} \cdot (\Box \times \{1\}).
\]
We recover $\Psi$ from $\D$ as 
\begin{equation}
   \label{eq:dualizing}
   \Box = \{u \mid (u,1) \in (\tail \D)^\vee\}; \qquad \Psi(u) = \D((u,1)).
 \end{equation}

\begin{definition}
\label{def:fano-polytope}
A divisorial polytope $\Psi$ is called \emph{Fano} if there is  an integral divisor $K_{\PP^1}=\sum_P a_P \cdot P$ of degree $-2$, such that 
\begin{enumerate}
  \item $0 \in \Box^\circ$, in particular $\deg \Psi(0) >0$;\label{item:0-interior}
  \item $\Psi_P(0) + a_P + 1 > 0$ and the facets of the graph of $(\Psi_P + a_P + 1)$ have lattice distance $1$ from the origin\label{item:distance-vertical};
  \item every facet $F$ of $\Box$ with $(\deg \Psi)|_F \not \equiv 0$ has lattice distance $1$ from the origin. \label{item:lattice-distance-horizontal}
\end{enumerate}
\end{definition}

\begin{example}
\label{exp:del-pezzo}
 We consider the Fano complete intersection in $\PP^4$ cut out by the equations
\[x_2^2 -x_2x_0+x_3x_4=0, \quad x_2^2-x_0x_1=0.\]
This is a del Pezzo surface of degree $4$ with three $A_1$ singularities. It admits a $\CC^*$-action induced by the weights $(0,0,0,1,-1)$. The corresponding Fano divisorial polytope lives over $\Box=[-1,1]$ and it is given by 
  \begin{equation*}
    \Psi_{0}(u) =\min\{0,-u\}+1, \quad \Psi_{1}(u) =\min\{0,u\}, \quad  \Psi_{\infty}(u) =\min\{u,-u\}+1
  \end{equation*}
   The corresponding graphs can be seen from Figure~\ref{fig:div-pol}.
  \begin{figure}[h]
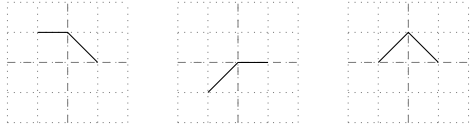

    \centering
\poldivdp
\caption{$\Psi_0$, $\Psi_1$ and $\Psi_\infty$.}
\label{fig:div-pol}
\end{figure}
One checks that $\Psi$ fulfills the Fano condition from Definition~\ref{def:fano-polytope} by choosing $K_{\PP^1} = -[0]-[\infty]$.
\end{example}

\begin{theorem}
 \label{thm:fano-divpols}
 Fano divisorial polytopes correspond to canonical Gorenstein Fano T-varieties.
\end{theorem}

The theorem follows from \cite[Section~3]{t3folds} by using the calculation of discrepancies from \cite[Section~4]{tsing}. For the convenience of the reader we give a more direct proof using the polyhedral divisor $\D$ for the section ring and the following Proposition by Schwede and Smith:
\begin{proposition}
\label{prop:schwede-smith}
  Let $(X,L)$ be a polarized variety. Then $X$ is Gorenstein, canonical and Fano and $L^{\otimes \ell} \cong \CO(-K_X)$ for some $\ell \in \NN$ if and only if the corresponding section ring is Gorenstein and canonical.

  In particular, $X$ is Fano and $L\cong \CO(-K_X)$ if and only if the section ring is Gorenstein and canonical, but none of its proper Veronese subrings are.
\end{proposition}
\begin{proof}
  This is the essence of (the proofs of) \cite[Prop. 6.4]{zbMATH01700888} and \cite[5.4]{zbMATH05712547}.
\end{proof}

\begin{proof}[Proof of Theorem~\ref{thm:fano-divpols}]
Consider the anti-canonical section ring $S$ for a Fano $T$-variety $X$ of complexity one. Then we obtain a p-divisor $\D$ over $\PP^1$ with lattice $N'=N \times \ZZ$ for the affine $T'$-variety $\spec S$, and a divisorial polytope $\Psi$ defined by (\ref{eq:dualizing}). We are going to prove that $\Psi$ fulfills the conditions of Definition~\ref{def:fano-polytope}.

Proposition~\ref{prop:schwede-smith} tells us that $\TV(\D)$ is Gorenstein. Because of the grading, this implies that a canonical divisor for $\TV(\D)$ is principal. Hence, for any $K_{\PP^1}=\sum_P a_P P$ being a canonical divisor on $\PP^1$, by \eqref{eq:canonical} there exists some $u'=(u,c)\in M\times\ZZ$, $f\in \CC(\PP^1)$, and such that
\[
\Div f\cdot \chi^{u'}=-\sum_\rho D_\rho + \sum_{P,v} (1-\mu(v))\cdot D_{P,v}  + \sum_{P,v} \mu(v) a_P  \cdot D_{P,v}.
\]
Using \eqref{eq:principal} and replacing $K_{\PP^1}$ by $K_{\PP^1}-\Div(f)$, we may assume that $f=1$. By \eqref{eq:principal}, we then have that
 for all extremal rays $\rho$ and all $P\in \PP^1$ and vertices $v\in \D_P$, 
\begin{align}
  \rho(u') &= -1\label{eq:gor-1}\\
  \mu(v)\langle u',v\rangle &=  -1+\mu(v) + \mu(v)a_P.\label{eq:gor-2}    
\end{align}
From this it is clear that the $|c|$-Veronese subring of $S$ is still Gorenstein, so we must have $c=-1$ by Proposition~\ref{prop:schwede-smith}.
Furthermore, after applying an automorphism of $M'=M\times \ZZ$ fixing the second factor, we may assume that $u'=(0,c)=(0,-1)$. This corresponds to a change of linearization of $L$.

Now, for every non-extremal ray $\rho$ in the tail cone (those which intersect $\deg \D$) we obtain discrepancies in a certain resolution of singularities for $\TV(\D)$ which are give by 
\[
   \discr_\rho=-1 -\rho(u'),
\]
see \cite[Section~4]{tsing}.
Since $S$ is canonical, we conclude that $\rho(u')=\rho(0,-1)<0$. Taking into account our condition on extremal rays above, we have $\rho(0,1)>0$ for all rays $\rho$, implying that $0$ is contained in the interior of $\Box = \{u \mid \forall_\rho:\rho(u,1)\geq 0\}$. Hence, condition (\ref{item:0-interior}) of Definition~\ref{def:fano-polytope} is fulfilled.

By \eqref{eq:dualizing} every facet $F$ of $\Box$ corresponds to a ray $\rho_F$ of $\tail \D$ such that 
\begin{equation}
  \label{eq:perp}
  \rho_F \perp \QQ_{\geq 0} (F \times \{1\}).
\end{equation}
Now, for $u \in F$ we have $\deg \Psi(u) = \deg \D(u,1) = \min \langle (u,1), \deg \D \rangle > 0$ if and only if $\rho_F$ does not intersect $\deg \D$. Hence, for $\deg \Psi(u)|_F \not \equiv 0$ the condition  \eqref{eq:gor-1} says that $\rho_F(0,-1)=-1$ and \eqref{eq:perp} is equivalent to the fact that $F$ has lattice distance $1$ from the origin. It follows that condition (\ref{item:lattice-distance-horizontal}) from Definition~\ref{def:fano-polytope} is fulfilled.

Again by \eqref{eq:dualizing} every facet $F$ of the graph $\Psi_P$ corresponds to a vertex $v_F'$ of $\D_P$. Indeed, for $(u,\Psi_P(u))\in F$ we have $\Psi(u) = \langle(u,1),v_F' \rangle$ for a unique vertex $v_F'$. Writing $v_F'=(v,h)$,  by \eqref{eq:gor-2} we obtain 
\[\-\mu(v_F)h=-1+\mu(v_F)(a_P+1) \]
and hence 
\[h=-\frac{\mu(v_F)(a_P+1)-1}{\mu(v_F)}.\]
We can thus write  $\mu(v_F)$ as $\ell \cdot \mu(v)$ for some $\ell\in\NN$. We obtain
\[\Psi(u) = \D((u,1)) = \langle (u,1),(v,h)\rangle = \langle u, v \rangle + \frac{1}{\ell \cdot \mu(v)} - a_P -1.\]
Note that the graph of $\langle \cdot , v \rangle + \frac{1}{\ell \cdot \mu(v)}$ has lattice distance $\nicefrac{1}{\ell}$ from the origin. But the integrality condition in Definition~\ref{def:divpol} implies that $\ell = 1$ and, hence, condition (\ref{item:distance-vertical}) of Definition~\ref{def:fano-polytope} is fulfilled and $\Psi$ is a Fano divisorial polytope. 

Conversely, if we start with a Fano divisorial polytope, it is straightforward to confirm the Gorenstein condition for the corresponding p-divisor by checking (\ref{eq:gor-1}) and $(\ref{eq:gor-2})$ with $u'=(0,-1)$. Since $0 \in \Box^\circ$,  $\tail \D = (\QQ_{\geq 0} (\Box \times \{1\}))^\vee$ is contained in $N_\QQ \times \QQ_{\geq 0}$. Now, for each ray $\rho$ of $\tail \D$, $\rho(0,-1)<0$, which implies $\discr_\rho=-1 -\rho(0,-1) > -1$, and hence $\TV(\D)$ is log-terminal. But since $\TV(\D)$ is also Gorenstein, this is the same thing as being canonical. 
\end{proof}

\begin{remark}\label{rem:psi0}
Let $\Psi$ be a Fano divisorial polytope such that $\Psi_P(0)\in\ZZ$ for all $P\in \PP^1$. By definition, we thus have that for some canonical divisor $\sum_P a_P\cdot P$, $\Psi_P(0)+a_P+1=1$ for all $P$. Indeed, since we always have
$\Psi_P(0)+a_P+1>0$, $\Psi_P(0)$ integral implies $\Psi_P(0)+a_P+1\geq 1$. But if it were strictly greater than $1$, any facet of the graph of $\Psi_P+a_P+1$ containing $(0,\Psi_P(0))$ would have lattice distance larger than one from the origin. 
We furthermore claim that the slopes of each $\Psi_P$ are all integral. Indeed, $\Psi_P+a_P+1$ is locally of the form
$u\mapsto \langle u,v \rangle + c$. By concavity, we have $c\geq 1$. Locally, the slope of $\Psi_P$ is given by $v$. Let $\mu$ be the smallest natural number such that $\mu\cdot v$ is in $N$. Then the height of the corresponding facet of the graph of $\Psi_P+a_P+1$ is $\mu c$. Hence, we must have $\mu=1$, and $\Psi_P$ has integral slope.
 \end{remark}

\subsection{Admissible Collections of Polyhedra} 
The following definition is will be important for determining the normality of the special fiber $\X_0$ of a special equivariant test configuration $(\X,\L)$.

\begin{definition}[{cf. \cite[Definition 2.2]{ilten:12a} }]
Let $I$ be a set, $\{ \Delta_i \}_{i\in I}$ a collection of rational polyhedra $\Delta_i\subset N_\RR$ with some common tail cone $\sigma\subset N_\RR$. This collection is \emph{admissible} if for all $u\in \sigma^\vee\cap M$, there is at most one $i$ for which $\Delta_i(u)$ is non-integral. 
Equivalently, this collection is admissible if for $u\in \sigma^\vee$ satisfying $\dim \face(\Delta_i,u)=0$ for all $i$, there is at most one $i$ for which $\face(\Delta_i,u)\notin N$.
\end{definition} 
\noindent Here, $\face(\Delta_i,u)$ denotes the face of $\Delta_i$ on which $u$ is minimal.

\section{Combinatorial description of test configurations}
\label{sec:test-configs}
\subsection{Combinatorial Construction}
\label{sec:comb-construction}
Let $(X,L)$ be a polarized complexity-one projective rational $T$-variety. We first describe a construction to produce certain test configurations for $(X,L)$. 
The section ring
\[
A=\bigoplus_i H^0(X,L^{\otimes i}).
\]
is $M'=\ZZ\times M$ graded, where $M$ is the character lattice of $T$, and hence may be described using a p-divisor $\D$ on $(\PP^1,N')$:
\[
A\cong \bigoplus_{u\in \sigma^\vee\cap M'} H^0(\PP^1,\CO(\lfloor \D(u)\rfloor)\cdot \chi^u.
\]
Here, $\sigma$ is the tailcone of $\D$, and $N'$ is the lattice dual to $M'$.

Choose a point $Q\in \PP^1$,   a natural number $m\in\NN$, and an element $v\in N'_\QQ$ such that $m\cdot v\in N'$. We define a polyhedral divisor $\widetilde \D$ with respect to the lattice $\widetilde{N}=N'\times \ZZ$:
\begin{align*}
\widetilde \D=\left(\conv\{\left((v,1/m)\right),(\D_Q,0)\}+\widetilde\sigma\right)\otimes Q+\sum_{\substack {P\in\PP^1\\P\neq Q}} \left( (\D_P,0)+\widetilde \sigma\right)\otimes P\\
\widetilde \sigma=\QQ_{\geq 0}\cdot \left\{(v,1/m)+\sum_{\substack {
P\in\PP^1\\P\neq Q}}(\D_P,0),\sum_{P\in\PP^1}(\D_P,0)\right\}.
\end{align*}
Set 
\[
\cA:=\bigoplus_{u\in \widetilde\sigma^\vee\cap \widetilde M} H^0(\PP^1,\CO(\lfloor \widetilde\D(u)\rfloor)\cdot \chi^u.
\]

\begin{proposition}\label{prop:construction}
The polyhedral divisor $\widetilde \D$ is proper, and 
\[
\X(Q,v,m):=\proj \cA
\]
is a $T$-equivariant test configuration for $X$. Here, the $\ZZ$-grading of $\cA$ used in $\proj$ is the first factor in $\widetilde M=\ZZ\times M\times \ZZ$, whereas the $\CC^*$-action of the test configuration corresponds to the one parameter subgroup generated by $(0,1)\in N'\times \ZZ$.
\end{proposition}
\begin{proof}
By construction $\deg (\widetilde \D)$ is contained in $\widetilde \sigma$, but the origin is clearly not contained in $\deg (\widetilde \D)$. Hence, $\widetilde \D$ is proper, see Section \ref{sec:pdiv}.

Via the $\proj$ construction, $\X(Q,v,m)$ comes equipped with a structure morphism to 
$\spec \cA_0$. For any weight $u\in \widetilde\sigma^\vee\cap \widetilde M$
of the form $u=(0,w,k)\in\ZZ\times M\times \ZZ$, we have $w=0$. Indeed, $(0,w)\in \sigma^\vee$ implies $w=0$. It follows that $k\geq 0$. But for such $u$, $\widetilde \D(u)=\D(0)=0$, so 
\[\cA_0=\CC[t]
\]
where $t=\chi^w$ has degree $w=(0,0,1)$. Hence, we have a $T\times \CC^*$-equivariant family $\X(Q,v,m)\to \A^1$ where $\CC^*$ acts on $\A^1$ in the standard fashion. Since $\X(Q,v,m)$ is irreducible, this family is flat and thus a test configuration as desired.
\end{proof}

We now wish to describe the special fiber (with the induced $\CC^*$ action) of such test configurations $\X(Q,v,m)$. To start out, consider the sublattice 
\[
\widetilde{M}_{m}=M'\times m\ZZ\subset \widetilde M.
\]
Let $\tau\subset \widetilde N_\RR$ be the cone generated by
\[
(0,1/m)+\sum_{\substack {
P\in\PP^1\\P\neq Q}}(\D_P,0)\qquad\textrm{and}\qquad (\D_Q,-1/m).
\]
Consider the semigroup 
\[
S=\{ (u,k)\in \widetilde{M}_{m}\cap \tau^\vee\ |\ \sum_{\substack {
P\in\PP^1\\P\neq Q }} \lfloor \D_P(u)\rfloor +k/m\geq 0\}.
\]
\begin{proposition}\label{prop:fiber}
Let $\X(Q,v,m)$ be as in Proposition \ref{prop:construction}, and $\X_0=\X_0(Q,v,m)$ the special fiber of this test configuration.
\begin{enumerate}
\item The special fiber $\X_0$ is isomorphic to the toric variety
\[
\proj \CC[S] 
\]
with induced $\CC^*$ action corresponding to the one-parameter subgroup generated by $(-mv,1)\in N'\times \frac{1}{m}\ZZ$.
\item The normalization of $\X_0$ is isomorphic to the toric variety 
\[
\proj \CC[\tau^\vee\cap \widetilde{M}_{m}]. 
\]
\item $\X_0$ is normal if and only if the set $\{\D_P\}_{P\neq Q}$ is admissible.
\end{enumerate}
\end{proposition}
\begin{proof}
Consider the lattice automorphism $\phi:\widetilde N=N'\times \ZZ \to \widetilde N$ which is the identity on $N'$ and sends $(0,1)$ to $(-mv,1)$. 
The pushforward of $\widetilde \D$ under $\phi$ is exactly the p-divisor for the test configuration $\X(Q,0,m)$.
 However, after this lattice automorphism, the one-parameter subgroup corresponding to the $\CC^*$-action on the original test configuration is given by $\phi((0,1))=(-mv,1)$. Hence, we may reduce to considering the special fiber of $\X(y,0,m)$ with a modified $\CC^*$-action.
In what follows, we will thus consider the p-divisor $\widetilde \D$ for the test configuration $\X(Q,0,m)$.

We now proceed in a similar fashion to the proof of \cite[Proposition 4.11]{hausen:10a}. With notation as above,
\[\X_0=\proj \cA/(t).\]
The degree $(u,k)$ piece of $B=\cA/(t)$ is simply
\[
H^0(\PP^1,\CO(\lfloor \widetilde\D(u,k)\rfloor))/H^0(\PP^1,\CO(\lfloor \widetilde\D(u,k-1)\rfloor))
\]
and hence is non-zero if and only if
\begin{enumerate}
\item $\lfloor \widetilde\D(u,k)\rfloor-\lfloor \widetilde\D(u,k-1)\rfloor>0$; and
\item $\widetilde\D(u,k))\geq 0$.
\end{enumerate}
Since we have
\begin{align*}
\lfloor \widetilde\D(u,k)\rfloor-\lfloor \widetilde\D(u,k-1)\rfloor=(\lfloor \min \{\D_Q(u),k/m\} \rfloor 
- \lfloor \min \{\D_Q(u),(k-1)/m\} \rfloor)\cdot Q
\end{align*}
the first condition implies that if $B_{(u,k)}\neq 0$ we must have $(u,k)\in \widetilde{M}_{m}$. Furthermore, $(u,k)$ must lie in the cone dual to that spanned the by right hand set of generators of $\tau$.
Satisfying the second condition implies that $(u,k)$ is in the cone dual to that spanned by the left hand set of generators of $\tau$. In fact, given $(u,k)\in \tau^\vee\cap \widetilde{M}_{m}$, satisfying the second condition is equivalent to being in the semigroup $S$.
We thus see that the projective coordinate ring of $\X_0$ has non-zero components exactly in the degrees of $S$. Each of these components a one-dimensional vector space and $\X_0$ is irreducible by construction; the first statement of the proposition follows.

For the second statement of the proposition, it is straightforward to check that $S$ generates the lattice $\widetilde{M}_{m}$ and that the weight cone of $S$ is exactly $\tau^\vee$. Hence, $\CC[\tau^\vee \cap \widetilde{M}_{m}]$ is the normalization of $\CC[S]$ and the second statement follows.

For the third statement, we first claim that $\X_0$ is normal if and only if $\CC[S]$ is normal, that is, is equal to $\CC[\tau^\vee\cap \widetilde{M}_{m}]$. Certainly $\CC[S]$ normal implies that $\X_0$ is normal. Conversely, note that $\CC[S]$ is isomorphic to $\bigoplus_i H^0(\X,\L^{\otimes i})_{|\X_0}$. Using the exact sequence
\[
\begin{CD}
0 @>>> \L^{\otimes i} @>>> \L^{\otimes i} @>>> \L_0 @>>> 0
\end{CD}
\] 
and the fact that $H^1(\X,\L^{\otimes i})=0$ for $i \gg 0$, we see that for $i\gg 0$, the degree $i$ graded piece of $\CC[S]$ is isomorphic to $H^0(\X_0,\L_0^{\otimes i})$. 
Since $\X_0$ is normal if and only if $\bigoplus_i H^0(\X_0,\L_0^{\otimes i})$ is normal, we that $\X_0$ normal implies $\CC[S]$ and $\CC[\tau^\vee \cap \widetilde{M}_m]$ agree in sufficiently high degree. But suppose that there exists an element $(u,k)\in \tau^\vee \cap \widetilde{M}_m$ with $(u,k)\notin S$, and let $\lambda\in \NN$ be such that $\lambda \D_z(u)\in\ \ZZ$ for all $z\neq y$.
Then 
\[
\sum_{z\neq y}  \D_z(( j\lambda+1) u)- \sum_{z\neq y} \lfloor \D_z(( j\lambda+1) u)\rfloor
=\sum_{z\neq y} \D_z( u)- \sum_{z\neq y} \lfloor \D_z(u)\rfloor\geq 1
\]
for any $j\geq 0$, so for any $j\geq 0$ there is $k'\in\ZZ$ such that
$( (j\lambda +1)u,k')$ belongs to $\tau^\vee\cap \widetilde{M}_m$ but not $S$. We conclude that if $\X_0$ is normal, $\CC[S]=\CC[\tau^\vee\cap \widetilde{M}_m$].

We have now reduced the third statement to showing that $\CC[S]$ is normal if and only if $\{\D_P\}_{P\neq Q}$ is admissible. On the one hand, if $\{\D_P\}_{P\neq Q}$ is admissible, we will always have 
\[
\sum_{P\neq Q}\D_P( u)- \sum_{P\neq Q}\lfloor \D_P(u)\rfloor <1
\]
for any $u\in M$, so $S=\tau^\vee \cap \widetilde{M}_m$. Let us suppose instead that $\{\D_P\}_{P\neq Q}$ is not admissible. Then  there exists some $u\in \sigma^\vee\cap M$ and points $P_1,P_2\in\PP^1$ such that neither $\D_{P_1}(u)$ nor $\D_{P_2}(u)$ are integral. After passing to some multiple of $u$, 
\[
\sum_{P\neq Q}\D_P( u)- \sum_{P\neq Q}\lfloor \D_P(u)\rfloor \geq 1.
\]
Furthermore, after potentially replacing this $u$ by a sufficiently large multiple of the form $(j\lambda+1)u$, we may assume that there is some $k\in \ZZ$ with $(u,k)\in\tau^\vee\cap \widetilde{M}_m$. If we take $k$ to be as small as possible, then similar to above $(u,k)\in\tau^\vee \cap\widetilde{M}_m$ but it does  not belong to $S$. Hence, $\CC[S]$ normal implies that $\{\D_P\}_{P\neq Q}$ is admissible.
\end{proof}

\begin{remark}
One could also calculate the normalization of the special fiber of $\X(y,v,m)$ using \cite[\S 7 and Theorem 10.1]{affine-tvar}. Indeed, the special fiber is the closure of the $T$-orbit corresponding to the vertex $(v,1/m)$ of $\widetilde \D_Q$, and hence is the toric variety 
for the cone 
\[\QQ_{\geq 0}\cdot \left(\widetilde \D_Q-\left(v,\frac{1}{m}\right)\right)\]
with respect to the lattice 
\[\widetilde N_m=N'\times \frac{1}{m}\ZZ.\]
But the lattice automorphism
\begin{align*}
\widetilde N_m&\to\widetilde N_m\\
(w,k)&\mapsto (w-mv,k)\in N'\times\frac{1}{m}\ZZ
\end{align*}
sends the above cone exactly to $\tau$.
\end{remark}

\begin{remark}
The degeneration $\X(Q,v,m)$ can also be understood as a family of complexity-one $T$-varieties by moving the coefficients $\D_P$ together for $P\neq Q$, as we now explain. This is not necessary for the sequel, but perhaps of independent interest. For notational purposes, assume that $Q=\infty\in \PP^1$.
 Consider the toric surface $S=\PP^1\times \A^1$ with $y_0,y_1$ coordinates on $\PP^1$ and $t$ a coordinate on $\A^1$.
For $\lambda \in \PP^1\setminus\{\infty\}$, we have the divisor
\[
D_\lambda=V(\lambda t^my_0- y_1)\subset S
\]
along with 
\begin{align*}
D_Q:=V(y_0)\\
D_v:=V(t).
\end{align*}
The polyhedral divisor 
\[
\E=\sum_{P\in\PP^1} \D_P\otimes D_P+(v+\sigma)\otimes D_v
\]
is $\CC^*$-invariant with respect to the action given by $\deg t=1$, $\deg (y_1/y_0)=m$. We leave it for the reader to check that this polyhedral divisor is in fact a p-divisor.

With respect to this $\CC^*$ action, the Chow quotient of $S$ is $\PP^1$. Hence, using the upgrade procedure of \cite[\S2]{upgrades}, the $T$-variety defined by $\E$ on $S$ can be described as a $T\times\CC^*$-variety by a p-divisor on $\PP^1$. We leave it for the reader to verify that the resulting p-divisor is exactly $\widetilde \D$.
Thus, as a $T$-variety, $\X$ is encoded by the p-divisor $\E$ on $S$.  Furthermore, the structure morphism $\X\to \A^1$ is exactly the map induced by the projection $\pi:S\to \A^1$.

To understand the fibers of $\X\to \A^1$, we can restrict the p-divisor $\E$ to the fibers of $\pi:S\to\A^1$, see \cite[\S 2]{ilten:12a} for a similar situation.
For $t\neq 0$, $\E_t:=\E_{|\pi^{-1}(t)}$ is just the p-divisor
\[
\E_t=\sum_{P\neq Q} \D_P\otimes (t^mP)+\D_Q\otimes Q
\]
and the corresponding fiber $\X_t$ is equivariantly isomorphic to $\TV(\D)$ via the automorphism of $\PP^1$ induced by $y_1/y_0\mapsto t^my_1/y_0$.
For $t=0$, 
\[
\E_0=\left(\sum_{P\neq Q} \D_P\right)\otimes (0)+\D_Q\otimes Q
\]
and this describes the normalization of the special fiber, cf. \cite[Theorem 2.8]{ilten:12a}.
Indeed, this is a $\CC^*$-invariant p-divisor on $\PP^1$, and the upgrade procedure tells us that after adding an additional $\CC^*$-action, the resulting $T\times\CC^*$-variety is exactly the toric variety $\TV(\tau)$, where $\tau$ is as in Proposition \ref{prop:fiber}.

 In any case, we observe that this degeneration of $\TV(\D)$ occurs by collapsing the coefficients $\D_P$ of $\D$ together for $P\neq Q$. Note that (as we have observed in Proposition \ref{prop:fiber}) the lattice point $v$ has no effect on the special fiber itself, and only changes the induced extra $\CC^*$-action on it.
\end{remark}

\subsection{Special Test Configurations}
We now  show that many $T$-equivariant test configurations are of the above form:
\begin{theorem}
Let $(X,L)$ be a polarized complexity-one projective rational $T$-variety, and $(\X,\L)$ a $T$-equivariant test configuration with $\X$ normal. Assume that the special fiber $\X_0$ is reduced and irreducible, and $(\X,\L)$ is not a product configuration.
Then $\X$ is of the form $\X(y,v,m)$ from Proposition \ref{prop:construction}.
\end{theorem}
\begin{proof}
Consider the $\CC$-algebra
\[
\cA = \bigoplus_i H^0(\X,\L^{\otimes i}).
\]
Then we obtain $\X = \proj \cA$. The $T\times \CC^*$ action on $(\X,\L)$ induces an effective grading of $\cA$ by 
\begin{align*}
\widetilde M=M'\times \ZZ,\\
M'=\ZZ\times M,
\end{align*}
 where $M$ is the character lattice of $T$, the $\ZZ$-factor of $M'$ corresponds to the standard $\ZZ$-grading on $\cA$ as in the direct sum decomposition above, and the $\ZZ$-factor of $\widetilde M$ is induced by the $\CC^*$-action of the test configuration.
Since $X$ is rational and $\X$ is normal, $\cA$ is both rational and normal. Furthermore, we have $\cA_{(0,0,0)}=\CC$. Hence, there exists a p-divisor $\widetilde \D$ on $\PP^1$ with tailcone $\widetilde \sigma\subset \widetilde{N}_\RR$ such that

\[
\cA \cong \bigoplus_{u\in \widetilde\sigma^\vee \cap \widetilde M} H^0(\PP^1,\CO(\lfloor \widetilde \D(u)\rfloor))\cdot \chi^u.
\]

By construction, the degree zero piece $\cA_0$ of $\cA$ with respect to the standard grading is of the form $\CC[t]$, where $t$ has degree $w=(0,1)\in M'\times \ZZ$ in the $\widetilde M$ grading.
If we write $t=s\cdot \chi^{w}$ for some rational function $s\in\CC(\PP^1)$, we can replace $\widetilde \D$ by the p-divisor 
\[\widetilde \D+\left( (0,1)+\widetilde\sigma\right)\otimes \Div (s).
\]
This new p-divisor gives rise to a $\CC$-algebra with an $\widetilde M$-graded isomorphism to $\cA$. 
Indeed, the $\CC$-algebra for the original p-divisor is isomorphic to that for the corrected  p-divisor via the map
$f\cdot \chi^{(u,k)}\mapsto s^{-k} f \cdot \chi^{(u,k)}$.
Hence, after making such a substitution, we can assume that $\widetilde \D(kw)$ is trivial for all $k\in\NN$, and $t=\chi^w$.

The special fiber $\X_0$ of $\X$  is given by $\proj \cA/(t)$. Since $\X_0$ is reduced and irreducible, it is a prime divisor in $\X$ and the inclusion $(\X \setminus \X_0) \subset \X$ is given by the localization in $t$, that is, $\X \setminus \X_0 = \proj \cA_t$. 
Now, $\cA_t$ corresponds to the p-divisor
\[
\D'=:\emptyset\otimes (\Div(t/\chi^w)+\widetilde \D(w))+\sum_{P\in \PP^1} \face(\widetilde \D_P,w)\otimes P
\]
by \cite[Proposition 3.3]{ahs}. But by assumption, $(\Div(t/\chi^w)+\widetilde \D(w))$ is trivial, hence
$\D'$ has no $\emptyset$-coefficient. In addition, the minimal value of $w$ on each $\widetilde \D_P$ equals zero, since $\widetilde\D(kw)$ is trivial for all $k\in \NN$. Hence we have
\[
\D'=\widetilde \D\cap w^\perp=\sum_{P\in\PP^1}\left( \widetilde \D_P\cap w^\perp\right)\otimes P.
\]
Since $(\X \setminus \X_0)$ is equivariantly isomorphic to the product $X \times \CC^*$, it follows that the p-divisor corresponding to the section ring of $L$ is also given by $\widetilde \D\cap w^\perp$, but this time viewed as a p-divisor with respect to the lattice $N'=\widetilde N\cap w^\perp$ instead of $\widetilde N$.
We denote this p-divisor by $\D$.

Note that $\TV(\D')$ differs from $\TV(\widetilde \D)$ by one exactly one $(\CC^*\times T \times \CC^*)$-invariant prime divisor $H$, corresponding to the special fiber $\X_0$. In the terminology of \cite{petersen} there are two possibilities: $H$ is either \emph{horizontal} or \emph{vertical}.
Suppose first that  $H$ is horizontal. Then the coefficients of $\widetilde \D$ and $\D'$ have the same vertices. In particular, all vertices of $\widetilde \D$ lie in $w^\perp$.  
Furthermore, $\widetilde \sigma=\tail(\widetilde \D)$ has exactly one extremal ray $\rho$ which is not an extremal ray of $\sigma'=\tail(\D')$.
But since the vertices of the coefficients of $\widetilde \D$ and $\D'$ have the same vertices, any non-extremal ray of $\sigma'$ is necessarily non-extremal in $\widetilde \sigma$. Thus, $\rho$ cannot be a ray of $\sigma'$ at all.
 Note that any ray of $\widetilde\sigma$ not contained in $w^\perp$ is necessarily extremal. Likewise, any ray of $\widetilde\sigma$ contained in $w^\perp$ is already in $\sigma'$. Hence, every non-extremal ray of $\widetilde\sigma$ is already in $\sigma'$.
We conclude that $\widetilde\sigma=\sigma'+\rho$, and thus
 $\widetilde\D=\sum(\D_P+\rho)\otimes P$, for some ray $\rho\subset \widetilde N_\RR$ with primitive lattice generator $(v,k)\in N'\times \ZZ$, where $k\neq 0$.
By \eqref{eq:principal}, the multiplicity of the divisor $D_\rho$ corresponding to $\rho$ in $\Div(t)=\Div(\chi^w)$ is $k$, so we conclude that $k=1$, and without loss of generality $\rho$ is generated by $(0,1)$. This implies that $\X$ is a product configuration. 

Suppose instead that $H$ is vertical.
 In the vertical case there is exactly one point $Q \in \PP^1$ with the vertices of $\widetilde \D_Q$ not equal to those of $\D_Q'$, and for this $Q$ we have exactly one additional vertex, say $(v,k)\in N_\QQ'\times \QQ$. Again by \eqref{eq:principal}, the multiplicity of the divisor $D_{Q,(v,k)}$ corresponding to $(v,k)$ in $\Div(t)=\Div(\chi^w)$ is $k\cdot \mu$, where $\mu$ is the smallest natural number such that $\mu\cdot (v,k)$ is in $N$.
As above, we must thus have $k\cdot \mu=1$, so we have $k=1/m$, where $m\cdot v\in N'$.

Now, the data of $\D$, $Q$, $v\in N_\QQ'$, and $m\in \NN$ completely determine $\widetilde\sigma=\tail \widetilde \D$, and hence $\widetilde \D$. Indeed, we have that
\[\Delta:=\conv(\deg \D',(v,1/m) + \sum_{P \neq Q} \D'_P) \subset \deg \widetilde \D.\]
Hence, the positive hull $C$ of $\Delta$ is in $\widetilde\sigma$. 
Note also that
\begin{equation}\label{eqn:deg} 
\deg \widetilde \D=\Delta+\widetilde\sigma.
\end{equation}
We claim that $C$ is in fact equal to $\widetilde\sigma$. Suppose that $\rho\in \widetilde\sigma$ is a ray not contained in $C$. 
Then by Equation \eqref{eqn:deg}, $\rho\cap \deg\widetilde \D=\emptyset$, and $\rho$ is an extremal ray. But this is impossible, since $\TV(\D')$ differs from $\TV(\widetilde D)$ by the vertical divisor $H$.
We conclude that indeed $\widetilde \D$ is determined by $\D$, and $(Q,v,m)$. Comparing with Proposition \ref{prop:construction}, we see that the test configuration of $\X$ is equivariantly isomorphic to $\X(Q,v,m)$, since $\D_P'=(\D_P,0)$.
\end{proof}

\subsection{Dual Description}
It will be useful to rephrase some of the results of this section in terms of divisorial polytopes. Let $\Psi:\Box\to \wdiv_\QQ\PP^1$ be the divisorial polytope corresponding to the polarized toric variety $(X,L)$.
\begin{corollary}
\label{cor:special-fiber}
The normalization $\widetilde \X_0$ of the special fiber of $\X=\X(Q,v,m)$ is the polarized toric variety given by the polytope
\[\Delta_Q = \Big\{(u,a) \in M_\QQ\otimes \RR \; \Big| \; u \in \Box,\; -\sum_{P \neq Q} \Psi_P(u) \leq a \leq \Psi_Q(u)\Big\}\]
with respect to the lattice $M\times \ZZ$.
The induced $\CC^*$-action on $\widetilde \X_0$ is given by the one-parameter subgroup $(-m\overline v,m) \in  N\times \ZZ $, where $\overline v$ is the projection of $v\in N_\QQ'$ to $N_\QQ$. 
\end{corollary}
\begin{proof}
  By construction, the cone $\tau$ from Proposition \ref{prop:fiber}  describes the normalization $\widetilde{X_0}$ of $\X_0$, with polarization induced by $\L$. Hence, the polarized toric variety $(\widetilde{X}_0, \L_{\widetilde{X}_0})$ is given by the polytope 
\[
\Delta_Q'  \quad = \quad \tau^\vee \cap (\{1\}\times M_\RR\times \RR).
\]
with respect to the lattice $M\times m\ZZ$. 
Taking the description of $\tau$ into account, we have
\[
\Delta_Q'  = \Big\{(u,a) \in M_\QQ\otimes \RR \; \Big| \; u \in \Box,\; -\sum_{P \neq Q} m\Psi_P(u) \leq a \leq m\Psi_Q(u)\Big\}.
\]
The polytope $\Delta_P$ above is simply the image of $\Delta_P'$ under the map induced by the lattice isomorphism $\phi:M\times m\ZZ \to M\times \ZZ$ sending $(u,a)$ to $(u,a/m)$.
Again by Proposition \ref{prop:fiber}, the $\CC^*$ action on $\X_0$ is the one-parameter subgroup given by the image under $\phi^*$ of the projection of $(-mv,1)$ to $N\times \frac{1}{m}\ZZ$; this is exactly $(-m\overline v, m)\in N\times \ZZ$.
\end{proof}

\begin{remark}
From the above corollary, we see that the $\CC^*$-action on $\widetilde \X_0$ is primitive if and only if $m$ is the smallest natural number such that $mv\in N$. In fact, from the point of view of $K$-stability, the above corollary shows that among test configurations of the form $\X(Q,v,m)$, we only need to consider those with $m=1$ and $v\in N$.
\end{remark}

\begin{corollary}\label{cor:normal1}
The special fiber of $\X(Q,v,m)$ is normal if and only if for each $u\in \Box$, $\Psi_P$ has  non-integral slope at $u$ for at most one $P \neq Q$. 
\end{corollary}
\begin{proof}
Consider a generic $u\in\Box$. If $\D$ is the p-divisor corresponding to $\Psi$, note that $\Psi_P$ has non-integral slope at $u$ if and only if the vertex $\face(\D_P,(u,1))$ is not a lattice element.
In other words, the collection of polyhedra $\{\D_P\}_{P\neq Q}$ is admissible if and only if for each $u\in \Box$, $\Psi_P$ has non-integral slope at $u$ for at most one $P\neq Q$.
 The claim is now a direct consequence of Proposition \ref{prop:fiber}. 
\end{proof}

\begin{corollary}\label{cor:normal-special-fiber}
Suppose that $X$ is Fano and canonical with $L$ the anticanonical divisor. Then the special fiber of $\X(Q,v,m)$ is normal if and only if $\Psi_P(0)$ is non-integral for at most one $P \neq Q$. 
\end{corollary}
\begin{proof}
This is a direct consequence of Corollary \ref{cor:normal1} and Remark \ref{rem:psi0}.
\end{proof}

\begin{example}
\label{exp:del-pezzo2}
We consider the divisorial polytope from Example~\ref{exp:del-pezzo}. When choosing $Q=\infty$ the corresponding test configuration $\X(Q,v,m)$ is given by the equations
\[x_2^2 -a\cdot x_0x_2+x_3x_4=0, \quad x_2^2-x_0x_1=0\]
with $\CC^*$-action being induced by the weights $-mv \cdot (-1,1,0,0,0)$ on $(x_0:\ldots:x_4)$ and $m$ on $a$. The special fiber for $a=0$ is the toric variety $\PP^1\!\!\times\!\PP^1/(\ZZ/2\ZZ)$ given by the polytope 
\[
\Delta_\infty = \conv\{(-1,0),(1,0),(0,-1),(0,1)\}.
\]
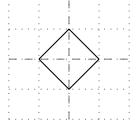
\begin{center}
\polytopedp
\end{center}
\end{example}

\subsection{Calculating the Futaki invariant}
Let $X$ be a Fano variety. Since the central fiber of a special test configuration $\X$ is $\QQ$-Fano, it follows that the polytope corresponding to $-K_{\X_0}$ contains a unique interior lattice point. For our polytope $\Delta_Q$, it is not difficult to see that this point is given by $u_Q=(0, -a_Q-1) \in M \times \ZZ$, where $K_{\PP^1} = \sum_P a_P \cdot P$ was the divisor of degree $-2$ coming with the Fano divisorial polytope $\Psi$. Let us denote by $b_Q$ the barycenter of $\Delta_Q^0 := \Delta_Q - u_Q$.

\begin{corollary}
\label{cor:df-formula}
For an  equivariant special test configuration $\X = \X(Q,v,m)$ the Donaldson-Futaki invariant is given by
 \[\DF(\X,\L)= \langle b_Q, (-mv,m) \rangle\]
and the modified Donaldson-Futaki invariant by
\[\DF_\xi(\X,\L)=\frac{m}{\vol \Delta_Q^0}\left(\int_{\Delta_Q^0} \langle \cdot, (-v,1) \rangle \cdot e^{\langle \cdot, \xi\rangle} \; d\mu\right).\]
\end{corollary}
\begin{proof}
Both claims follow from the equality 
\[F_{\X_0, \xi}(v') = \frac{1}{\vol \Delta_Q^0}\left(\int_{\Delta_Q^0} \langle \cdot, v' \rangle \cdot e^{\langle \cdot, \xi\rangle}\right),\]
for the toric special fiber $\X_0$, which can be found in the proof of \cite[Prop. 2.1]{wang14}. The right-hand-side appears there as a formula for a coefficient $F_0$, which by Proposition.~1.1 of loc.~cit.~agrees with our definition of $F_{\X_0, \xi}(v)$.
\end{proof}

This results in the following combinatorial criterion for equivariant K-stability for complexity-one Fano $T$-varieties:

\begin{theorem}
\label{thm:k-stability}
  A Fano $T$-variety corresponding to a divisorial polytope $\Psi$ is equivariantly K-stable if and only if for every $Q \in \PP^1$ such that there is at most one $P \neq Q$ with $\Psi_P(0)$ non-integral, we have 
  \[b_Q \in \{0\} \times \RR^+ \subset M_\RR \times \RR.\]
\end{theorem}
\begin{proof}
Let $p_1$ and $p_2$ be the projections from $M_\RR \times \RR$  to the first and second factor, respectively. 
We have $\langle p_1(b_Q), v \rangle = F_{\X_0}(v)$ using the canonical inclusion $N_\RR \rightarrow N_\RR \times \RR$. On the other hand the flatness of $\L$ over $\A^1$ implies the same for the homogeneous components $\L_i$. In particular, we have $F_{\X_0}(v)=F_X(v)$. In other words $p_1(b_Q)=0$ is sufficient and necessary for having non-negative Donaldson-Futaki invariant for trivial test-configurations.

Assuming now that $p_1(b_y)=0$, Corollary~\ref{cor:df-formula} implies that $p_2(b_Q)>0$ is equivalent to the positivity of $\DF(\X(Q,v,m),\L)$ for all choices of $v$ and $m$.
\end{proof}

\begin{remark}
  Note, that $b_Q$ is the same for every choice $Q \notin \supp \Psi$. Hence, the criterion is indeed effective. 
\end{remark}

\begin{example}
For the divisorial polytope from Example~\ref{exp:del-pezzo}. We would have to check the above condition for $\Delta_0$, $\Delta_1$ and $\Delta_\infty$. As we observed in Example~\ref{exp:del-pezzo2} we obtain $\Delta_\infty = \conv\{(-1,0),(1,0),(0,-1),(1,0)$.
The barycenter is $(0,0)$. We conclude that the test configuration destabilizing, since the Donaldson-Futaki invariant of the test configuration is $0$. On the other hand we have $F_X = p_2(b_\infty)=0$. Hence, the del Pezzo surface does not admit a K\"ahler-Ricci soliton.
\end{example}

\section{Additional symmetries}
\label{sec:addit-symmetries}
In this section, we aim to further reduce the number of test configurations we must consider by utilizing the action of the normalizer $G=\Aut_T(X)$ of the torus $T$ in $\Aut(X)$. By \cite[Lemma~2.9]{t3folds} this group is reductive. Then, by Theorem~\ref{thm:gabor} it is enough to check those test configurations which are $G$-equivariant. This is in particular useful if there are no non-trivial test configurations left to consider, in which case we immediately obtain the existence of a K\"ahler-Ricci soliton.

Let $X$ be a rational complexity-one $T$-variety.
An element $\varphi$ of $\Aut_T(X)$ is an equivariant automorphism of $X$. This implies that there is an automorphism of the character lattice $F:M' \rightarrow M'$ such that semi-invariant rational functions of degree $u \in M'$ are send to semi-invariant functions of degree $F(u)$ by this automorphism. In particular, invariant functions are send to invariant functions and we obtain a natural action of $\Aut_T(X)$ on the quotient $\PP^1$. This action descends to an action of $W=\Aut_T(X)/T$.
By \cite[Proposition~2.7]{t3folds} for $X=\TV(\D)$ an element of $W$ corresponds to a pair $(\psi,F) \in \Aut(\PP^1) \times \Aut(M')$ such that for each $P\in\PP^1$ there exists $v_P \in N'$ satisfying
\begin{enumerate}
\item All but finitely many $v_P$ are zero;
\item $\sum_P v_P = 0$; and
\end{enumerate}
 \begin{equation}
  \label{eq:automorphism}
  \forall_{P \in \PP^1}:F(\D_{P}) + v_y= \D_{\psi(P)}.
\end{equation}
The action of  $\overline{\varphi} \in \Aut_T(X)/T$ on $\PP^1$ is then induced by $\psi \in \Aut(\PP^1)$ from the above pair.

As before we consider a Fano $T$-variety $X$ corresponding to a divisorial polytope $\Psi$; the spectrum of the anti-canonical ring $A$ is given by $\TV(\D)$ for some polyhedral divisor $\D$ over the lattice $N'=\ZZ \times N$. Now, a non-trivial $T$-equivariant special test configuration $\X(Q,v,c)$ is given by the polyhedral divisor $\widetilde \D$ over the lattice $N' \times \ZZ$ with $\D = \widetilde \D \cap N'$ as in Section~\ref{sec:comb-construction}.

A pair $(\psi,F)$ for $\D$ corresponds an element of  $\Aut_{T'}(\spec A)$ which descends to an element of $\Aut_T X$ with $X = \proj A$ if the second component of $N'=  N\times Z$ is fixed under $F$. Given such a pair, the corresponding automorphism extends to a test configuration $\X(Q,v,c) = \TV(\widetilde \D)$ as above if and only if 
\[\forall_{P \in \PP^1}:\widetilde F(\widetilde \D_{P}) + v_Q= \widetilde \D_{\psi(P)}\]
holds. Here we consider the inclusion $v_Q \in N' \hookrightarrow \widetilde N = N' \times \ZZ$ and $\widetilde F$ is a linear extensions of $F$ to an automorphism of $\widetilde N$. Since by construction $\widetilde \D_Q$ in contrast to all other polyhedra $\widetilde \D_P$ has an vertex outside of $N'\subset \widetilde N$ the point $Q$ must be a fixed point of $\psi$.

For $P\in\PP^1$, let $\mu_P$ the maximal denominator of a vertex in $\D_P$ or equivalently the maximal denominator of a slope of $\Psi_P$.

\begin{theorem} 
\label{thm:soliton-criterion}
Let $X$ be a smooth Fano $T$-variety of complexity one as above. Assume that
  \begin{enumerate}
  \item There three points $P_1, P_2, P_3 \in \PP^1$ with $\mu_{P_1}, \mu_{P_2}, \mu_{P_3} > 1$; or \label{item:3points}
  \item There are two points $P_1, P_2 \in \PP^1$ with $\mu_{P_1}, \mu_{P_2} > 1$ which are swapped by an element of $\Aut_T(X)$; or \label{item:2points}
  \item $\Aut_T(X)$ acts fixed-point-free on $\PP^1$.\label{item:fp-free}
  \end{enumerate}
Then $X$ admits a K\"ahler-Ricci soliton.
\end{theorem}
\begin{proof}
  First note that by \cite{tian02} there always exist a unique solution $\xi \in N_\RR$ such that $F_{X,\xi}=0$. Hence, for this choice  of $\xi$, the condition for trivial test configurations is fulfilled. Now, in the above three situations we will show that there do not exist non-trivial $\Aut_T(X)$-equivariant special test configurations. Indeed, this is obvious for (\ref{item:3points}), since the condition of Corollary~\ref{cor:normal-special-fiber} is violated. For (\ref{item:fp-free}) the claim follows from the fact that we have to choose $Q$ to be a fixed point of $\Aut_T(X)$ in order to preserve the $\Aut_T(X)$-action on the test configuration. Similarly in the situation of (\ref{item:2points}) we have to choose $Q$  to be a fixed point, but this implies $Q \neq P_1$ and $Q \neq P_2$. Hence, again the normality conditions of Corollary~\ref{cor:normal-special-fiber} is violated.
\end{proof}

By using the rational quotient map $\pi: X \dashrightarrow \PP^1$ given by the identity $\CC(\PP^1) \cong \CC(X)^T$ and the fact that $\mu_P$ encodes the order of a finite stabilizer on the fiber $\pi^{-1}(P)$ (cf. \cite[Prop. 4.11]{tcox}) 
we may reformulate Theorem~\ref{thm:soliton-criterion} as follows:

\begin{theorem} 
\label{thm:soliton-criterion2}
Let $X$ be a smooth Fano $T$-variety of complexity one. Assume that
  \begin{enumerate}
  \item $T$ acts with disconnected stabilizers on three fibers of $\pi$; or \label{item:3points2}
  \item $T$ acts with disconnected stabilizers on two fibers of $\pi$ which are swapped by an element of $\Aut_T(X)$; or \label{item:2points2}
  \item $\Aut_T(X)$ does not fix any fiber of $\pi$.\label{item:fp-free2}
  \end{enumerate}
Then $X$ admits a K\"ahler-Ricci soliton.
\end{theorem}

We obtain Theorem~1.1 of \cite{kesym} as a corollary, since the notion of symmetry established there implies the vanishing of the Futaki invariant $F_X$. Hence, the vector field corresponding to the soliton has to be the trivial one, and we obtain a K\"ahler-Einstein metric.

To check the hypotheses of Corollary~\ref{thm:soliton-criterion} directly on a divisorial polytope, we have to relate the condition on the pair $(\psi,F')$ for $\D$ to a dual condition for the divisorial polytope $\Psi$. Indeed, one checks  using \eqref{eq:dualizing} in Section~\ref{sec:divisorial-polytopes} that there is an element of $\Aut_T(X)$ sending $P_1 \in \PP^1$ to $P_2 \in \PP^1$ if and only if there is a pair $(\psi, F^*)$ consisting of an automorphism of $\PP^1$ and an automorphism of $M$ such that $F^*(\Box)=\Box$ and
\begin{equation}
  \label{eq:automorphism-dual}
  (\Psi_P \circ F^*) + v_P + b_P = \Psi_{\psi(y)} + b_{\psi(P)}
\end{equation}
for some $v_P \in N = M^*$, $b_P\in\ZZ$, with $\sum_P v_P =\sum_P b_P = 0$. Here, $v_P$ is understood as a linear form on $M$ and $b_P$ as a constant function.

\begin{example}
\label{exp:threefold}
Consider the blowup of $\PP^1 \times \PP^2$ in a curve of degree $(1,2)$. This is threefold 3.21 from Mori and Mukai's list \cite{mori-mukai}. By \cite{t3folds} this threefold admits a 2-torus action and the corresponding divisorial polytope is given by Figure~\ref{fig:fano-3fold}. Here, the piecewise linear functions are given by the regions of linearity and the values at the vertices.
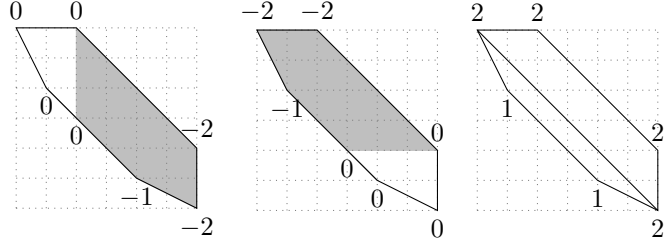
\begin{figure}[h]
  \centering
\begin{tikzpicture}[scale=0.4]
\draw[color=lightgray, fill] (-1,0) node[color=black,anchor=north] {$0$} -- (-1,3) -- (-1,3) -- (3,-1) -- (3,-3) -- (1,-2) -- (-2,1) -- (-1,0);
\draw[dotted,step=1,gray] (-3,-3) grid (3,3);
\draw (-2,1) node[anchor=north] {$0$}-- (-3,3) node[anchor=south] {$0$} -- (-1,3)node[anchor=south] {$0$} -- (3,-1) node[anchor=south] {$-2$}-- (3,-3) node[anchor=north] {$-2$}-- (1,-2) node[anchor=north] {$-1$}-- (-2,1);
\end{tikzpicture}
\begin{tikzpicture}[scale=0.4,y={(1cm,0cm)},x={(0cm,1cm)}]
\draw[color=lightgray, fill] (-1,0) node[color=black,anchor=north] {$0$} -- (-1,3) -- (-1,3) -- (3,-1) -- (3,-3) -- (1,-2) -- (-2,1) -- (-1,0);
\draw[dotted,step=1,gray] (-3,-3) grid (3,3);
\draw (-2,1) node[anchor=north] {$0$}-- (-3,3) node[anchor=north] {$0$} -- (-1,3)node[anchor=south] {$0$} -- (3,-1) node[anchor=south] {$-2$}-- (3,-3) node[anchor=south] {$-2$}-- (1,-2) node[anchor=north] {$-1$}-- (-2,1);
\end{tikzpicture}
\begin{tikzpicture}[scale=0.4]
\draw[dotted,step=1,gray] (-3,-3) grid (3,3);
\draw (-2,1) node[anchor=north] {$1$}-- (-3,3) node[anchor=south] {$2$} -- (-1,3)node[anchor=south] {$2$} -- (3,-1) node[anchor=south] {$2$}-- (3,-3) node[anchor=north] {$2$}-- (1,-2) node[anchor=north] {$1$}-- (-2,1);
\draw (-3,3) -- (3,-3);
\end{tikzpicture}
  \caption{The divisorial polytope for 3.21: $\Psi_0$, $\Psi_\infty$ and $\Psi_1$ }
  \label{fig:fano-3fold}
\end{figure}
Now, consider the pair $\left(x \mapsto \nicefrac{1}{x},
\left(\begin{smallmatrix}
  0&1\\
  1&0
\end{smallmatrix}
\right)\right)$. It fulfills the condition \eqref{eq:automorphism-dual} just with all $v_P,b_P$ being equal to zero.

Note, that $\mu_0=\mu_\infty=2$ since over the grey shaded regions the slope is $(-\nicefrac{1}{2},0)$ or $(0,-\nicefrac{1}{2})$, respectively. Hence, the conditions of Theorem~\ref{thm:soliton-criterion} are fulfilled and we obtain the existence of a K\"ahler-Ricci soliton on the threefold.

Let us now show how to obtain this result by using Theorem~\ref{thm:soliton-criterion2} instead. We have a 2-torus action on $\PP^1 \times \PP^2$ given by the weights $\left(\begin{smallmatrix}-2& 0\quad& 1& 0& 0 \\ 0&-2\quad &0&1&0\end{smallmatrix}\right)$ on the coordinates $(u:v,\quad x:y:z)$. The corresponding quotient map $\pi:\PP^1 \times \PP^2 \dashrightarrow \PP^1$ is given by \[(u:v,\quad x:y:z) \mapsto (ux^2:vy^2).\]
The invariant prime divisors $[x=0]$ and $[y=0]$ are contained in $\pi^{-1}(0)$ and $\pi^{-1}(\infty)$, respectively. From the weight matrix we see that on both prime divisors the generic stabilizer is of order $2$. Moreover, there is an involution on $\PP^1 \times \PP^2$ given by $x \mapsto y$ and $u \mapsto v$ which swaps the fibers $\pi^{-1}(0)$ and $\pi^{-1}(\infty)$. Let us consider now the curve $[z=ux^2-vy^2=0]$. It is invariant under the torus action and under the involution. Hence, both actions survive the blowup in this curve, which gives us the threefold  3.21 as above.  The new quotient map is just the composition of $\pi$ with the blowup. Now, we can see directly by Theorem~\ref{thm:soliton-criterion2}~\eqref{item:2points2} that this manifold admits a K\"ahler-Ricci soliton without referring to its combinatorial description.
\end{example}

\section{Applications to threefolds} 
\setcounter{footnote}{0}
\begin{theorem}\label{thm:threefold}
    The Fano threefolds Q, 2.24\savefootnote{one-of-family2}{This refers only to a particular element of the family admitting a 2-torus action}, 2.29, 2.32, 3.10\repeatfootnote{one-of-family2}, 3.19, 3.20, 4.4 and 4.7 from the classification of Mori and Mukai \cite{mori-mukai} admit a K\"ahler-Einstein metric.
\end{theorem}
\begin{proof}
  We just apply Theorem~\ref{thm:k-stability} to the list of threefolds in \cite{t3folds}.
\end{proof}

\begin{theorem}\label{thm:soliton}
  The Fano threefolds 3.8\repeatfootnote{one-of-family2}, 3.21, 4.5\repeatfootnote{one-of-family2} admit a non-trivial K\"ahler-Ricci soliton.
\end{theorem}
\begin{proof}
  Using the combinatorial data given in \cite{t3folds} we apply Theorem~\ref{thm:soliton-criterion} (\ref{item:3points}) to the threefold 3.8  and (\ref{item:2points}) to the threefolds 3.21 and 4.5, see also Example~\ref{exp:threefold}.
\end{proof}

\begin{remark}
  Note that 3.8 and 4.5 are non-rigid in contrast to toric Fano manifolds, which by the results of Wang and Zhu \cite{wang04} are an important source of examples of non-trivial K\"ahler-Ricci soliton.
\end{remark}

\bibliographystyle{alpha}
\bibliography{kstab}
\end{document}